\documentclass[reqno0,12pt]{amsart}
\pdfoutput=1
\usepackage{dsfont}
\usepackage{bbm}
\usepackage[nodate]{datetime}
\usepackage[hmargin=25mm,top=28mm,bottom=28mm,a4paper]{geometry}
\usepackage{color}
\usepackage{subfig}
\usepackage{fancyhdr}
\usepackage{amsmath,amssymb,amsfonts,amsthm}
\usepackage{amstext}
\usepackage{amsmath}
\usepackage{amssymb}
\usepackage{enumerate}
\usepackage{amsbsy}
\usepackage{amsopn}
\usepackage{bbm,amsthm}
\usepackage{amscd}
\usepackage{amsxtra}
\usepackage{todonotes}

\newtheorem{theorem}{Theorem}[section]
\newtheorem*{theorem*}{Theorem}
\newtheorem{lemma}{Lemma}[section]
\newtheorem{proposition}{Proposition}[section]
\newtheorem{corollary}{Corollary}[section]

\theoremstyle{remark}

\newtheorem{example}{Example}[section]
\newtheorem{remark}{Remark}[section]

\newcommand{\CC}{\mathds{C}}

\newcommand{\NN}{\mathds{N}}

\begin{document}
\title{Complex valued multiplicative functions with bounded partial sums}
\author{Marco Aymone}

\begin{abstract} We present a class of multiplicative functions $f:\mathbb{N}\to\mathbb{C}$ with bounded partial sums. The novelty here is that our functions do not need to have modulus bounded by $1$. The key feature is that they pretend to be the constant function $1$ and that for some prime $q$, $\sum_{k=0}^\infty \frac{f(q^k)}{q^k}=0$. These combined with other conditions guarantee that these functions are periodic and have sum equal to zero inside each period. Further, we study the class of multiplicative functions $f=f_1\ast f_2$, where each $f_j$ is multiplicative and periodic with bounded partial sums. We show an omega bound for the partial sums $\sum_{n\leq x}f(n)$ and an upper bound that is related with the error term in the classical Dirichlet divisor problem.
\end{abstract}
\maketitle
\section{Introduction.}
We say that $f:\NN\to\CC$ is \textit{multiplicative} if $f(nm)=f(n)f(m)$ whenever $n$ and $m$ are relatively prime, and we say that such $f$ is \textit{completely multiplicative} if this relation holds for all $n$ and $m$. Therefore, a multiplicative function $f$ is determined by its values at prime powers. 

We say that $f:\NN\to\CC$ has bounded partial sums if there exists a constant $C>0$ such that for all $x\geq 1$, $|\sum_{n\leq x}f(n)|\leq C$; otherwise we say that $f$ has unbounded partial sums.

Resolving the Erd\H{o}s discrepancy problem,  Tao \cite{taodiscrepancy} showed that a complex valued completely multiplicative function $f$ with $|f|=1$ has unbounded partial sums. Further, Tao gave a partial classification of all multiplicative functions $f$ taking only values $\pm 1$ with bounded partial sums. To state this partial classification, we need to introduce the language of pretentious number theory \cite{granvillepretentious}: Given two complex valued multiplicative functions $f$ and $g$ taking values in the unit disk, we say that $f$ \textit{pretends} to be $g$ or that $f$ is $g$-\textit{pretentious} if the ``distance'' between $f$ and $g$ given by
\begin{equation*}
\mathbb{D}(f,g;x):=\left(\sum_{p\leq x}\frac{1-Re(f(p)\overline{g(p)})}{p}\right)^{1/2}
\end{equation*} 
is $O(1)$ as $x\to\infty$, where in the sum above $p$ stands for a generic prime number.

The multiplicative function $f:\NN\to\{-1,1\}$ such that $f(2^k)=-1$ for all $k\geq 1$ and $f(p^k)=1$ for all primes $p\geq 3$ and all powers $k\geq 1$, then $f$ is the periodic function $f(n)=(-1)^{n+1}$ which clearly has bounded partial sums. In \cite{taodiscrepancy}, Tao showed that if $f:\NN\to\{-1,1\}$ is multiplicative and has bounded partial sums, then $f$ is $1$-pretentious and at powers of $2$, $f(2^k)=-1$ for all $k\geq 1$. Later, Klurman \cite{klurmancorrelation} completely classified such multiplicative functions with bounded partial sums by proving that they must be periodic of some period $m$ and $\sum_{n=1}^m f(n)=0$. This last result had been known as the Erd\H{o}s-Coons-Tao conjecture.

When we allow that a multiplicative function $f$ takes complex values, then there is no known criterion to determine when $f$ has bounded partial sums, therefore we must analyze case by case. For instance, in \cite{aymonediscrepancy} and \cite{klurmanteravainen} it has been proved that a multiplicative function $f$ supported on the squarefree integers such that at primes $f(p)=\pm1$, has unbounded partial sums. On the other hand, without any restriction we can easily construct examples of multiplicative functions $f:\NN\to\CC$ with bounded partial sums. A non-trivial way to construct such examples exists if we impose conditions on the values $f(p)$ at primes $p$ such that $\sum_{n\leq x}|f(n)|$ is bounded below by $cx$ for all sufficiently large $x$, for some positive constant $c$. Here we aim to do this.

\begin{theorem}\label{teorema condicao necessaria soma zero} Assume that $f:\NN\to\CC$ is multiplicative, has bounded partial sums and $\sum_{p}\frac{|1-f(p)|}{p}<\infty$. Then there exists a prime $q$ such that
\begin{equation}\label{equacao sum zero}
\sum_{k=0}^\infty \frac{f(q^k)}{q^k}=0.
\end{equation}
\end{theorem}

\begin{remark} We impose the condition $\sum_{p}\frac{|1-f(p)|}{p}<\infty$ to keep the intuiton behind pretentiouness in the case that $f$ takes values outside the unit disk. For example, if $p_n$ is the $n$-th prime and $f(p_n)=1+(-1)^n\in\{0,2\}$, the partial sums $\sum_{p\leq x}\frac{1-f(p)}{p}$ are $O(1)$ as $x\to\infty$ while the values $f(p)$ are always distant from $1$.
\end{remark}

It is interesting to observe that if $f$ is real-valued and $f^2\leq 1$, then, since $f(1)=1$, \eqref{equacao sum zero} can only be satisfied when $q=2$ and $f(2^k)=-1$ for all $k\geq 1$. But we have many options to satisfy \eqref{equacao sum zero} when we allow that $f$ takes complex values. 

\begin{theorem}\label{teorema exemplos} If a multiplicative function $f:\NN\to\CC$ has period $m$, $f(m)\neq 0$ and has bounded partial sums, then the following three conditions are satisfied.\\
i. For some prime $q|m$, $\sum_{k=0}^\infty \frac{f(q^k)}{q^k}=0.$\\
ii. For each $p^a\| m$, $f(p^k)=f(p^a)$ for all $k\geq a$.\\
iii. For each $\gcd(p,m)=1$, $f(p^k)=1$, for all $k\geq 1$.\\
Conversely, if $f:\NN\to\CC$ is multiplicative and the three conditions above are satisfied, then $f$ has period $m$ and has bounded partial sums.
\end{theorem}

An intermediate step in the proof of the Erd\H{o}s-Coons-Tao conjecture \cite{klurmancorrelation} is a result similar to Theorem \ref{teorema exemplos}  -- Proposition 4.4 of \cite{klurmancorrelation}, where it is assumed that $f^2\leq 1$. Our contribution here is the observation that the proof of Proposition 4.4 of \cite{klurmancorrelation} allow us to deal with the case where $|f|$ is not necessarily bounded by $1$.

 We stress that the condition that $f$ does not vanish at its period, $f(m)\neq 0$, is pivotal to deduce the three conditions above. Indeed, a non-principal Dirichlet character is a classical example of a periodic (completely) multiplicative function with bounded partial sums that vanishes at its period, and does not satisfy either \textit{i}. and \textit{iii}. However, the three conditions above allow us to produce examples of periodic multiplicative functions with bounded partial sums, despite the fact that $f$ vanishes or not at its period.

\begin{example}
Let $f$ be multiplicative and define for all primes $p\neq 3$, $f(p^k)=1$ for all powers $k\geq 1$, and at powers of $3$: $f(3)=2$, $f(9)=-15$ and $f(3^k)=0$ for all $k\geq 3$. Then $f$ has period $27$, $f(27)=0$, and has bounded partial sums.
\end{example}

\begin{example} Let $f$ be multiplicative and define for all primes $p\neq 5$, $f(p^k)=1$ for all powers $k\geq 1$, and at powers of $5$: $f(5)=\pi$, $f(5^k)=-20-4\pi$ for all $k\geq 2$. Then $f$ has period $25$, $f(25)\neq0$, and has bounded partial sums.

\end{example}

We point out that our class of examples in Theorem \ref{teorema exemplos} is not the only one with bounded partial sums. Indeed we can construct very easily examples of non-periodic multiplicative functions with bounded partial sums by a standard convolution argument: If $g:\NN\to\CC$ is multiplicative and $\sum_{n=1}^\infty |g(n)|<\infty$, and if $h:\NN\to\CC$ has bounded partial sums, then $f=g\ast h$ also has bounded partial sums, where $\ast$ stands for Dirichlet convolution. In particular, $h$ can be as in Theorem \ref{teorema exemplos} or a non-principal Dirichlet character $\chi$.
 
Now we turn our attention to multiplicative functions $f:\NN\to\CC$ of the form $f=f_1\ast f_2$, where each $f_j$ is multiplicative and periodic with bounded partial sums. We begin by observing that if each $f_j$ satisfies the conditions i-iii of Theorem \ref{teorema exemplos}, then $f$ has unbounded partial sums.  

Before we state our next result, we recall the notation $f(x)=\Omega(g(x))$, where $g(x)>0$ for all $x>0$. This means that $\limsup_{x\to\infty}\frac{|f(x)|}{g(x)}>0$.
 
\begin{theorem}\label{teorema omega} Let $f_1$ and $f_2$ be two multiplicative functions satisfying conditions i-iii of Theorem \ref{teorema exemplos}. Let $f=f_1\ast f_2$. Then there exists a constant $d>0$ such that
$$\sum_{n\leq x}f(n)=\Omega\left(\exp\left(d \frac{\log x}{\log\log x} \right) \right). $$
\end{theorem}

A key argument in the proof of the result above is that $f(n)=\tau(n)$ whenever $\gcd(n,m)=1$ for some $m$, where $\tau(n)$ is the divisor function: $\tau(n)=\sum_{d|n}1$. The omega result is then obtained by using classical estimates for the maximal value of $\tau$.

Our next question concerns upper bounds for the partial sums of $f=f_1\ast f_2$ as in Theorem \ref{teorema omega}. We begin by recalling the classical estimate for the partial sums of the divisor function:

\begin{equation}\label{equacao somas parciais da tau}
\sum_{n\leq x}\tau(n)=x\log x+(2\gamma-1)x+\Delta(x),
\end{equation}
where $\Delta(x)$ is an error term and $\gamma$ is the Euler-Mascheroni constant.

In the past 200 years there were a lot of attempts to obtain sharp estimates for the error term $\Delta(x)$. This is classically known as the \textit{Dirichlet divisor problem}, where one seeks to obtain estimates for the exponent 

\begin{equation}\label{equacao alpha}
\alpha:=\inf\{a>0:\Delta(x)=O_a(x^a)\},
\end{equation}
where the notation $O_a$ means that the implied constant may depend on the parameter $a$.

It is common knowledge that $\alpha\geq 1/4$ (Hardy \cite{hardydirichlet} and Landau, independently), but its exactly value is unknown. It is conjectured that $\alpha=1/4$. The best upper bound up to date is due to Huxley \cite{Huxley2003} (2003): $\alpha\leq 131/416 \approx 0.314$. For a nice historical account on this problem we refer to the book of Tenenbaum \cite{tenenbaumlivro}.

Before we state our next result we recall some classical notation. Here $\mu$ is the M\"obius function.

\begin{theorem}\label{teorema cota superior} Let $f_1$ and $f_2$ be two multiplicative functions satisfying conditions i-iii of Theorem \ref{teorema exemplos}, and let $m_1$ and $m_2$ be the periods of $f_1$ and $f_2$, respectively. Let $f=f_1\ast f_2$. Then, for all $x>m_1m_2$
$$\sum_{n\leq x}f(n) = \sum_{n|m_1m_2}f\ast\mu\ast\mu(n)\Delta\left(\frac{x}{n}\right).$$
\end{theorem}

\begin{corollary}\label{corolario cota superior} Let $f$ be as in Theorem \ref{teorema cota superior} and $\alpha$ defined by \eqref{equacao alpha}. Then, for all $\epsilon>0$
$$\sum_{n\leq x}f(n)=O_\epsilon(x^{\alpha+\epsilon}).$$
\end{corollary}
In particular, by the result of Huxley:
$$\sum_{n\leq x}f(n)=O_\epsilon(x^{131/416+\epsilon}).$$

Thus, we have a considerable gap between our omega result (Theorem \ref{teorema omega}) and our upper bound above. Even for the simplest case $f_1(n)=f_2(n)=(-1)^{n+1}$ it seems to be hard to obtain sharp estimates for the partial sums of $f=f_1\ast f_2$. We speculate that $\sum_{n\leq x}f(n)=\Omega(x^{1/4})$, and in the final section of this paper we prove this omega bound for some particular cases of non-vanishing periodic multiplicative functions. We also discuss a possible approach to the general case.

\section{Proofs of the main results}

\subsection{Notation} We use both $f(x)\ll g(x)$ and $f(x)=O(g(x))$ whenever there exists a constant $C>0$ such that $|f(x)|\leq C|g(x)|$ for all  $x$ in a set of parameters. When not specified, this set of parameters will be the range in which $x$ is sufficiently large. Further, $\ll_\delta$ means that the implicit constant may depend on $\delta$. The standard $f(x)=o(g(x))$ means that $\lim_{x\to a}\frac{f(x)}{g(x)}=0$. Sometimes $a$ can be $\infty$. We write $\mathcal{P}$ for the set of primes and $p$ for a generic element of $\mathcal{P}$. The notation $p^k\| n$ means that $k$ is the largest power of $p$ for which $p^k$ divides $n$. Dirichlet convolution is denoted by $\ast$. 

\subsection{Proof of Theorems \ref{teorema condicao necessaria soma zero} and \ref{teorema exemplos}}
We begin with the following.
\begin{lemma}\label{lema f eh O(1)} If $f:\NN\to\CC$ is multiplicative and has bounded partial sums, then $\sup_n|f(n)|<\infty$ and for each $\epsilon>0$, there exists a $M>0$ such that if $p\geq M$, then $|f(p^k)|\leq 1+\epsilon$, for all $k\geq 1$.
\end{lemma}
\begin{proof} Let $C>0$ be such that $\left|\sum_{n\leq x} f(n) \right|\leq C$ for all $x\geq 1$. Assume by contradiction that $f$ is not $O(1)$. Thus there exists a sequence of integers $x_k\to \infty$ such that $|f(x_k)|\to\infty$. Since
$$|f(x_k)|-\left|\sum_{n\leq x_k-1}f(n) \right|\leq \left|\sum_{n\leq x_k}f(n) \right|\leq C,$$
we obtain a contradiction for large $k$. Thus $f$ must be $O(1)$. Now if there are an infinite number of distinct primes $p_1,p_2,...$ such that for some powers $k_1,k_2,...$, $|f(p_j^{k_j})|>1+\epsilon$, then $|f(n_l)|$ become arbitrarily large for $n_l=p_1^{k_1}\cdot...\cdot p_l^{k_l}$, and thus $f$ is not $O(1)$.
\end{proof}
\begin{proof}[Proof of Theorem \ref{teorema condicao necessaria soma zero}] Assume that $f$ has bounded partial sums. Therefore, the Dirichlet series $F(s):=\sum_{n=1}^\infty\frac{f(n)}{n^s}$ is analytic in the half plane $Re(s)>0$. By Lemma \ref{lema f eh O(1)} above there exists a constant $C>0$ such that $|f(n)|\leq C$,  and hence, for $Re(s)>1$, $F(s)$ is given by the Euler product
$$F(s)=\prod_{p\in\mathcal{P}}\sum_{k=0}^\infty \frac{f(p^k)}{p^{ks}}.$$
Now we split the Euler product in primes below and above $M$, where $M$ is such that for all primes $p\geq M$, $|f(p^k)|\leq 1+\epsilon$ for all $k\geq 1$. For the tail product
we have that for $\sigma>1$
$$\frac{1}{\zeta(\sigma)}\prod_{p>M}\sum_{k=0}^\infty \frac{f(p^k)}{p^{k\sigma}}= \prod_{p\leq M}\left(1-\frac{1}{p^\sigma}\right)\prod_{p>M}\left(1+\frac{f(p)-1}{p^{\sigma}}+\frac{O(1)}{p^{2\sigma}} \right).$$
Therefore, since we assume that $\sum_{p}\frac{|1-f(p)|}{p}<\infty$, by making $\sigma\to 1^+$ above we conclude that the limit exists, and since $\zeta(\sigma)=\frac{1}{\sigma-1}+O(1)$, there exists a constant $c\in\CC\setminus\{0\}$ such that
$$\prod_{p>M}\sum_{k=0}^\infty \frac{f(p^k)}{p^{k\sigma}}=\frac{c+o(1)}{\sigma-1},$$
as $\sigma\to 1^+$. Thus, as $F$ is analytic at $s=1$, we conclude that as $\sigma\to1^+$, the finite product
$$\prod_{p\leq M}\sum_{k=0}^\infty \frac{f(p^k)}{p^{k\sigma}}=O(\sigma-1),$$
and hence
$$\prod_{p\leq M}\sum_{k=0}^\infty \frac{f(p^k)}{p^{k}}=0,$$
and this can happen only if some Euler factor equals to $0$.
\end{proof}
The proof of the next result follows the lines of Proposition 4.4 of \cite{klurmancorrelation}.
\begin{proof}[Proof of Theorem \ref{teorema exemplos}] Assume that $f:\NN\to\CC$ is multiplicative, has period $m$, $f(m)\neq 0$ and has bounded partial sums. Then for all $k\geq 1$, $f(km)=f(m)$. In particular, since $f(m)\neq 0$, for each $k$ coprime with $m$, $f(k)=1$. Now write $m$ as a power of distinct primes, say $p_1^{a_1},...,p_l^{a_l}$, where each $a_j\geq 1$. Since $f(m)\neq 0$, we obtain that each $f(p_j^{a_j})\neq 0$. Thus, by setting $k=p_j^t$, the equation $f(km)=f(m)$ implies that $f(p_j^{a_j+t})=f(p_j^{a_j})$. Thus we have shown that conditions ii-iii are satisfied. 

 Observe that, since $f$ has period $m$ and bounded partial sums, we have that $\sum_{n\leq m}f(n)=0$.

Now notice that if $\gcd(n,m)=d$, then $f(n)=f(d)$. This is because for each $p^a\|n$ such that $\gcd(p,m)=1$, we have that $f(p^a)=1$, and if $p^b\|m$ with $b\geq 1$, we have that $f(p^a)=f(p^c)$ where $c=\min(a,b)$. Thus we can write
\begin{align*}
\sum_{n\leq m}f(n)=\sum_{d|m}\sum_{\substack{n\leq m\\\gcd(n,m)=d}}f(n)=\sum_{d|m}f(d)\varphi(m/d)=f\ast \varphi(m),
\end{align*}
where $\varphi$ is the Euler's totient function. Since $f$ and $\varphi$ are multiplicative, we have that $f\ast \varphi$ is multiplicative. Recall that $\varphi(p^a)=p^a(1-1/p)$. Thus for each $p^a\|m$ with $a\geq 1$, we have that
\begin{align*}
f\ast\varphi(p^a)&=f(p^a)+f(p^{a-1})p\left(1-\frac{1}{p}\right)+f(p^{a-2})p^2\left(1-\frac{1}{p}\right)+...+p^a\left(1-\frac{1}{p}\right)\\
&=p^a\left(1-\frac{1}{p}\right)\left(\sum_{k=0}^{a-1}\frac{f(p^k)}{p^k}+\frac{f(p^a)}{p^a(1-1/p)} \right).
\end{align*}
But since $f(p^a)=f(p^k)$ for all $k\geq a$, we have that
$$\frac{f(p^a)}{p^a(1-1/p)}=\sum_{k=a}^{\infty}\frac{f(p^k)}{p^k}.$$
Thus, 
\begin{equation}\label{equacao calculo da soma no periodo}
\sum_{n\leq m}f(n)=\varphi(m)\prod_{p|m}\sum_{k=0}^{\infty}\frac{f(p^k)}{p^k},
\end{equation}
and hence condition i. must be satisfied.

Now assume conditions i-iii. Then as above, if $\gcd(a,m)=d$, then $f(a)=f(d)$, and if $n\equiv a \mod m$, then $\gcd(n,m)=\gcd(a,m)$, and hence $f$ has period $m$. Now with conditions ii-iii we can arrive at \eqref{equacao calculo da soma no periodo}, and with condition i. we conclude that $\sum_{n\leq m}f(n)=0$, and thus $f$ must have bounded partial sums. \end{proof}

\subsection{Proof of Theorems \ref{teorema omega} and \ref{teorema cota superior} }
\begin{lemma}\label{lemma propriedades da g} Let $f=f_1\ast f_2$ where $f_1$ and $f_2$ are multiplicative functions satisfying conditions i) ii) and iii)  of Theorem \ref{teorema exemplos}. Let $m_1$ and $m_2$ be the periods of $f_1$ and $f_2$ respectively. Then $f=g\ast \tau$, where $g$ satisifies the following properties.\\
a) $\sum_{n\leq x}|g(n)|=O_\epsilon(x^\epsilon)$, for all $\epsilon>0$;\\
b) If $\gcd(n,m_1m_2)=1$, then $g(n)=0$;\\
c) $\sum_{n=1}^\infty \frac{g(n)}{n}=\sum_{n=1}^\infty \frac{g(n)\log n}{n}=0$.
\end{lemma}
\begin{proof} Let $Re(s)>1$. By the classical identity for the Dirichlet series of a convolution and the Euler product formula, we have that 
$$F(s):=\sum_{n=1}^\infty \frac{f_1\ast f_2(n)}{n^s}=\prod_{p\in\mathcal{P}}\left(\sum_{k=0}^\infty \frac{f_1(p^k)}{p^{ks}}\right)\left(\sum_{k=0}^\infty \frac{f_2(p^k)}{p^{ks}}\right).$$
Now, by assumption iii., if $\gcd(p,m_1m_2)=1$, then $f_1(p^k)=f_2(p^k)=1$ for all powers $k\geq 1$. Therefore
$$F(s)=\prod_{p|m_1m_2}\left(\sum_{k=0}^\infty \frac{f_1(p^k)}{p^{ks}}\right)\left(\sum_{k=0}^\infty \frac{f_2(p^k)}{p^{ks}}\right)\prod_{\substack{p\in\mathcal{P}\\ \gcd(p,m_1m_2)=1 } }\left(1-\frac{1}{p^s}\right)^{-2},$$
and hence
\begin{equation}\label{equacao formula para G}
G(s):=\frac{F(s)}{\zeta(s)^2}=\prod_{p|m_1m_2}\left(\sum_{k=0}^\infty \frac{f_1(p^k)}{p^{ks}}\right)\left(\sum_{k=0}^\infty \frac{f_2(p^k)}{p^{ks}}\right)\left(1-\frac{1}{p^s}\right)^{2}.
\end{equation}
Recall that $\zeta(s)^2$ is the Dirichlet series of $\tau = 1\ast 1$. Thus, $G(s)$ is the Dirichlet series of $g:=f\ast \tau^{-1}=f\ast(\mu\ast\mu)$, where $\mu$ is the classical M\"obius function. Therefore, by the Euler product formula for $G(s)$ above, we have that condition b) must be satisfied. Since $f_1$ and $f_2$ are $O(1)$, we have that there exists a constant $c>0$, such that for all primes $p$ and all powers $k\geq 1$, $|g(p^k)|\leq ck$. This implies that for each $\sigma>0$
$$\sum_{p|m_1m_2}\sum_{k=1}^\infty \frac{|g(p^k)|}{p^{k\sigma}}<\infty,$$
and hence, by a classical result for Dirichlet series (see for instance \cite{tenenbaumlivro}, pg. 188, Theorem 1.3), $G(s)=\sum_{n=1}^\infty \frac{g(n)}{n^s}$ converges absolutely in the half plane $Re(s)>0$ and is given by \eqref{equacao formula para G} for each $s$ in this half plane. In particular, for each $\epsilon>0$, $\sum_{n=1}^\infty\frac{|g(n)|}{n^\epsilon}<\infty$, and hence, by Kroenecker's Lemma (see for instance \cite{shiryaev}, pg. 390, Lemma 2), we have that condition a) is satisfied. Finally, by assumption i., there are primes $q_1|m_1$ and $q_2|m_2$ such that
$$\sum_{k=0}^\infty \frac{f_1(q_1^k)}{q_1^{k}}=\sum_{k=0}^\infty \frac{f_2(q_2^k)}{q_2^{k}}=0.$$
Hence, by analyticity
$$\sum_{k=0}^\infty \frac{f_1(q_1^k)}{q_1^{ks}}=O(|s-1|),\;\sum_{k=0}^\infty \frac{f_2(q_2^k)}{q_2^{ks}}=O(|s-1|),$$
for all $s$ sufficiently close to $1$. This combined with \eqref{equacao formula para G} gives that $G(s)=O(|s-1|^2)$, for all $s$ sufficiently close to $1$, and since $G$ is analytic, we have that $G(1)=G'(1)=0$. But
$G(1)=\sum_{n=1}^\infty \frac{g(n)}{n}$ and $G'(1)=-\sum_{n=1}^\infty \frac{g(n)\log n}{n}$, and this completes the proof.
\end{proof}

\begin{proof}[Proof of Theorem \ref{teorema omega}] By the triangle inequality we have that for each positive integer $x$,
$$|f(x)|\leq \left|\sum_{n\leq x -1}f(n) \right|+\left|\sum_{n\leq x }f(n) \right|.$$
Therefore, by the pigeonhole principle, we have that at least one of the two sums in the right-hand side above is at least  $|f(x)|/2$. By Lemma \ref{lemma propriedades da g}, we have that for each $\gcd(n,m_1m_2)=1$, $f(n)=\tau(n)\geq 2^{\omega(n)}$, where $\omega(n)=\sum_{p|n}1$. Since $\omega(n)$ can be as large as $(1+o(1))\log n /\log\log n$ (see for instance \cite{tenenbaumlivro} pg. 113, Theorem 5.4), we complete the proof.
\end{proof}

\begin{proof}[Proof of Theorem \ref{teorema cota superior}] By Lemma \ref{lemma propriedades da g}, we have that $f=g\ast \tau$. This combined with \eqref{equacao somas parciais da tau} gives that
\begin{align*}
\sum_{n\leq x}f(n) &= \sum_{n\leq x}g(n)\sum_{m\leq x/n}\tau(m)=\sum_{n\leq x}g(n)\left(\frac{x}{n}\log (x/n)+(2\gamma-1)\frac{x}{n}+\Delta(x/n)\right)\\
&=x\log x\sum_{n\leq x}\frac{g(n)}{n}-x\sum_{n\leq x}\frac{g(n)\log n}{n}+(2\gamma-1)x\sum_{n\leq x}\frac{g(n)}{n}+\sum_{n\leq x}g(n)\Delta(x/n).
\end{align*}

We will show that each of the first three sums in the right hand side above vanishes for $x>m_1m_2$. This is the moment when we will use the condition ii) of Theorem \ref{teorema exemplos}. The condition states that if $p^{k_j}\|m_j$ for $j=1,2$, then $f_j(p^{t_j})=f_j(p^{k_j})$ for all $t_j\geq k_j$. Here we allow that $k_j=0$. Thus $f_j\ast\mu(p^{t_j})=f_j(p^{t_j})-f_j(p^{t_j-1})=0$, for all $t_j\geq k_j+1$. Since $g=(f_1\ast\mu)\ast(f_2\ast\mu)$, we have that $g(p^{t})=0$ for all $t\geq k_1+k_2+1$. This can be easily seen by the fact that each Euler factor in the Euler product representation of $\sum_{n=1}^{\infty}\frac{g(n)}{n^s}$ has the form 
$$\sum_{k=0}^{k_1}\frac{f_1\ast\mu(p^k)}{p^{ks}}\sum_{l=0}^{k_2}\frac{f_2\ast\mu(p^l)}{p^{ls}}.$$
Therefore, by b) of Lemma \ref{lemma propriedades da g}, if $n>m_1m_2$, we have that $g(n)=0$, and by c) of the same Lemma  we obtain the desired claim. 
\end{proof}

\begin{remark} We observe that without ii. of Theorem \ref{teorema exemplos}, we could prove a slighty more general statement but with weaker conclusions in comparison with the one obtained in Theorem \ref{teorema cota superior}. Indeed, this was done in a preprint version of this paper (see arXiv:2110.03401, v3). There we show that 
$$\sum_{n\leq x}f_1\ast f_2(n)=\sum_{n\leq x}g(n)\Delta(x/n)+O_\epsilon(x^{\epsilon}),$$
for any $\epsilon>0$. To establish this, one should impose a growth condition on each $f_j$ so that one could prove a) from Lemma \ref{lemma propriedades da g}.
\end{remark}

\begin{proof}[Proof of Corollary \ref{corolario cota superior}] By Theorem \ref{teorema cota superior}, we have that $\sum_{n\leq x}f(n)$ can be expressed as a finite linear combination of the functions $(\Delta(x/n))_n$. The proof is then an immediate consequence of the triangle inequality.
\end{proof}


\section{Discussion on $\Omega$ bounds for $f_1\ast f_2$}
As we point out in the introduction, there is a considerable gap between our omega bound (Theorem \ref{teorema omega}) and our upper bound (Corollary \ref{corolario cota superior}). Here we propose an approach to prove our conjecture that $\sum_{n\leq x}f_1\ast f_2(n)=\Omega(x^{1/4})$, where each $f_j$ is periodic with bounded partial sums that satisfy i., ii. and iii. of Theorem \ref{teorema exemplos}.

A nice result proved by Tong \cite{tong} states that
$$\int_{1}^X\Delta(x)^2dx=(A+o(1))X^{3/2},$$
where $A$ is the constant given by

\begin{equation}\label{equation tong}
A=\frac{1}{6\pi^2}\sum_{n=1}^\infty\frac{\tau(n)^2}{n^{3/2}}.
\end{equation}
Observe that if a function $\lambda(x)$ is $o(x^{1/4})$, then $\int_1^X|\lambda(x)|^2dx=o(X^{3/2})$. Therefore, this combined with Tong's result gives a second proof (different from Hardy's) that $\Delta(x)=\Omega(x^{1/4})$, and also allows us to prove the following result.

\begin{proposition} Let $q$ be a prime number and $f$ be the unique $q$-periodic multiplicative function with bounded partial sums, and such that $f(q)\neq 0$. Assume that $q\geq 5$. Then
$$\sum_{n\leq x}f\ast f(n)=\Omega(x^{1/4}).$$
\end{proposition}
\begin{proof} By Theorem \ref{teorema exemplos} we have that $f(q^k)=f(q)$ for all $k\geq 1$. Then, condition i. determines uniquely the value of $f(q)$:
$$1+f(q)\sum_{k=1}^\infty\frac{1}{q^k}=0.$$
Therefore, $f(q)=-(q-1)$. Now, for $k\geq 1$, by condition iii. of Theorem \ref{teorema exemplos}, we have that $f\ast\mu(p^k)$ is $0$ unless $p=q$. In this case we have that
$f\ast\mu(q^k)$ is: $0$ if $k\geq 2$, and $f(q)-1=-q$, if $k=1$. The Euler factor corresponding to $q$ in the Euler product representation of the Dirichlet series of $g:=(f\ast\mu)\ast(f\ast\mu)$ is
$$\left(1-\frac{q}{q^s}  \right)^{2}=1-\frac{2q}{q^s}+\frac{q^2}{q^{2s}}.$$
This immediately implies that $g(q)=-2q$ and $g(q^2)=q^2$. Hence, by Theorem \ref{teorema cota superior}, for $x>q^2$
$$\sum_{n\leq x}f\ast f(n)=\Delta(x)-2q\Delta(x/q)+q^2\Delta(x/q^2).$$
Now, by combining Tong's result \eqref{equation tong} with a simple change of variables, for any $n>0$ we obtain that
$$\|\Delta(x/n)\|_{L^2[1,X]}:=\left(\int_1^X \Delta(x/n)^2dx  \right)^{1/2}=\frac{\sqrt{A+o(1)}}{n^{1/4}}X^{3/4}.$$
The Cauchy-Schwarz and Minkowski inequalities for $L^2$ spaces imply that
\begin{align*}
\|\Delta(x)-2q\Delta(x/q)&+q^2\Delta(x/q^2)\|_{L^2[1,X]}\\
&\geq q^2\|\Delta(x/q^2)\|_{L^2[1,X]}-2q\|\Delta(x/q)\|_{L^2[1,X]}-\|\Delta(x)\|_{L^2[1,X]}\\
&\geq\sqrt{A+o(1)}X^{3/4}\left(\frac{q^2}{q^{2/4}}-\frac{2q}{q^{1/4}}-1 \right)\\
&\geq\sqrt{A+o(1)}X^{3/4}\left(q^{3/2}-2q^{3/4}-1 \right).\\
\end{align*} 
With standard calculus we can check that the function $\lambda(q):=q^{3/2}-2q^{3/4}-1$ is increasing for $q\geq 1$, and since $\lambda(5)=3.4929$, we have that $\lambda(q)>0$ for all $q\geq 5$. This shows that in this range of $q$,
$$\bigg{\|}\sum_{n\leq x}f\ast f(n)\bigg{\|}_{L^2[1,X]}\geq\sqrt{A+o(1)}\lambda(q)X^{3/4},$$
which gives the desired omega bound. \end{proof}

In our proof above, we see that this method does not work in the case that $q\in\{2,3\}$, since in these cases $\lambda(q)<0$. In particular, in the case $q=2$, the associated periodic multiplicative function is the classical $f(n)=(-1)^{n+1}$. For this particular case, in the next figure we plot the partial sums of $f\ast f$, and the numerics are in agreement with our conjecture. 

\begin{figure}[h]
\includegraphics[scale=0.4]{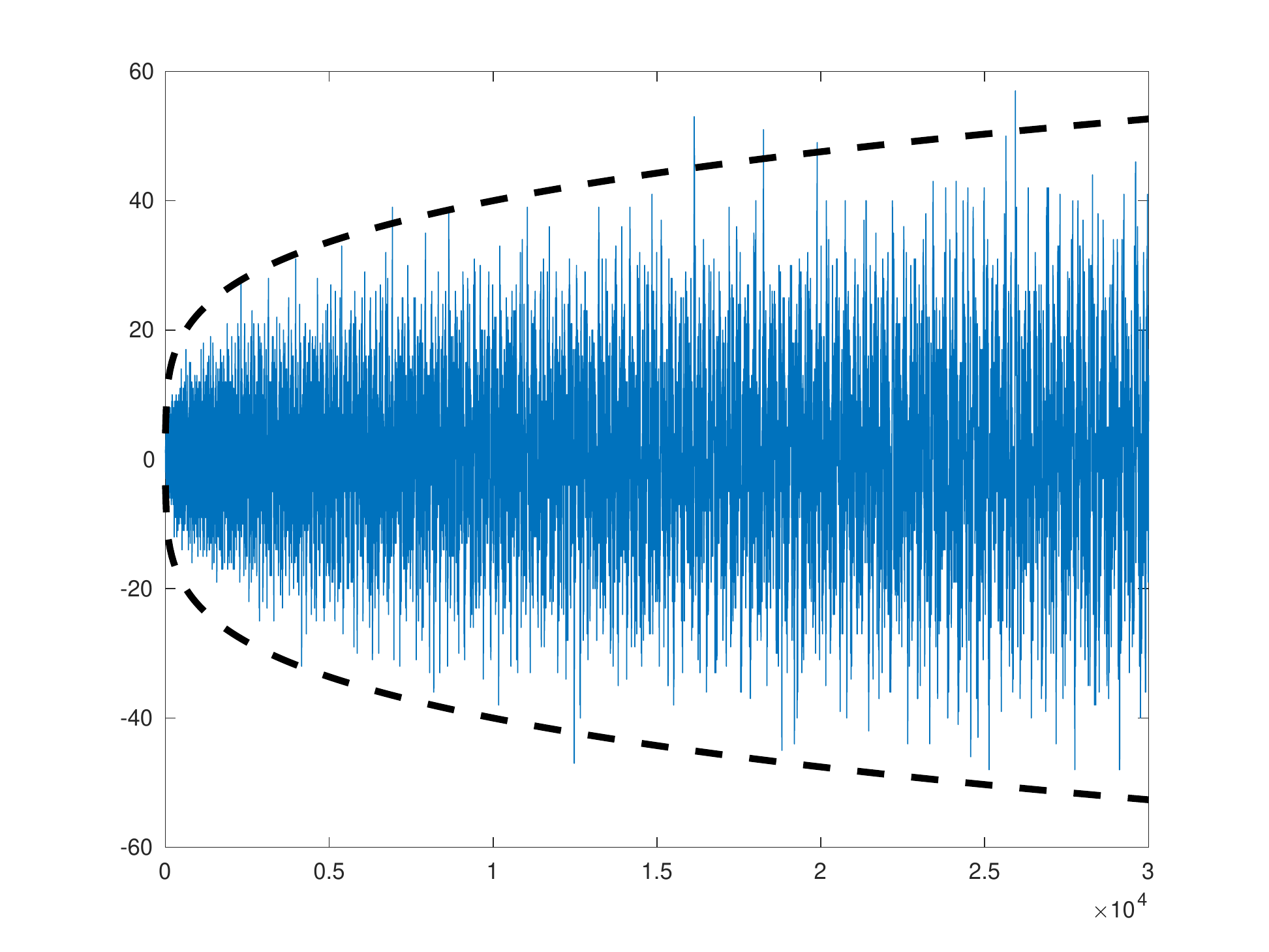} 
\caption{The dashed curves are given by $x\mapsto\pm 4x^{1/4}$, and the continuous line is given by $x\mapsto\sum_{n\leq x}f\ast f(n)$, where $f(n)=(-1)^{n+1}$.}
\end{figure}

We conclude by mentioning a possible approach to the conjectured omega bound. Since $\sum_{n\leq x}f_1\ast f_2(n)$ can be expressed as $\sum_{n\leq T}c_n\Delta(x/n)$, where $T$ is a positive integer and $c_n$ are complex numbers, one could approach the conjectured omega bound by studying the quadratic form obtained from the squared $L^2[1,X]$ norm of $\sum_{n\leq T}c_n\Delta(x/n)$. The conjectured omega bound would follow if, for instance, one could prove that the eigenvalues of the symmetric matrix $(a_{n,m})_{n,m\leq T}$ are positive,
where
$$a_{n,m}=\lim_{X\to\infty}\frac{1}{X^{3/2}}\int_{1}^X\Delta(x/n)\Delta(x/m)dx.$$
To prove this, our preliminary calculations show that, by using the classical Vorono\"i's formula for $\Delta(x)$ (see Lemma 1 of \cite{lautsang}), firstly we need to understand the effect of positive integers $a$ and $b$ in the correlations
$$\sum_{n\leq x}\tau(an)\tau(bn).$$

\noindent\textbf{Acknowledgements.}  I am thankful to Oleksiy Klurman for a fruitful discussion that resulted in this work. This was in February of 2020 while I was visisting MPIM (Bonn), for which I am also thankful for its warm hospitality. Also, I would like to thank the anonymous referee for several remarks and corrections that improved the exposition of this paper, and to my students Caio Bueno and Kevin Medeiros for their comments on a revised version of it. I am supported by CNPq - grant Universal number 403037/2021-2.

{\small{\sc \noindent
Departamento de Matem\'atica, Universidade Federal de Minas Gerais, Av. Ant\^onio Carlos, 6627, CEP 31270-901, Belo Horizonte, MG, Brazil.} \\
\textit{Email address:} aymone.marco@gmail.com}
\vspace{0.5cm}

\end{document}